%% file: qcHH-v1Mar.tex
\documentclass[10pt, a4paper, centertags, openright, mathscr]{amsart}
\usepackage[left]{lineno}

\usepackage{etex}
\usepackage{listings}

\usepackage{amsmath, amsfonts, amsthm, amssymb, amsxtra, bbm, enumerate, stmaryrd}
\usepackage[utf8]{inputenc}
\usepackage[english]{babel}
\usepackage[dvipsnames]{xcolor}

\usepackage[style=alphabetic,
            backend=biber,
            isbn=false,
            doi=false,
            url=false,
            hyperref=true, 
            firstinits=true,
            maxbibnames=99,
            maxalphanames=99,
            maxnames=99,
            block=none]{biblatex}
\renewbibmacro*{journal+issuetitle}{%
  \usebibmacro{journal}%
  \setunit*{\addspace}%
  \iffieldundef{series}
    {}
    {\newunit
     \printfield{series}%
     \setunit{\addspace}}%
  \printfield{volume}%
  \setunit{\addspace}%
  \usebibmacro{issue+date}%
  \setunit{\addcolon\space}%
  \usebibmacro{issue}%
  \setunit{\addcomma\space}%
  \printfield{number}%
  \newunit}

\renewbibmacro*{volume+number+eid}{%
\printfield{volume}
\printfield{number}
}

\renewbibmacro*{date}{\setunit{\addspace}\printdate}

\DeclareFieldFormat*{title}{\mkbibemph{#1}}%\addperiod % italic title with period
\DeclareFieldFormat*{journaltitle}{#1} % title of journal article is printed as normal text
\DeclareFieldFormat[incollection]{booktitle}{\rm #1} % title of journal article is printed as normal text
\DeclareFieldFormat*{volume}{vol. \bf #1}
\DeclareFieldFormat[article]{date}{#1}
\DeclareFieldFormat[book,inbook,incollection,thesis,unpublished]{date}{(#1)}
\DeclareFieldFormat*{number}{no. #1}
\DeclareFieldFormat*{pages}{#1}
\renewbibmacro{in:}{}

\usepackage{graphicx}
\usepackage[all, cmtip]{xy}
\xyoption{curve}

\usepackage{tikz}
\usetikzlibrary{arrows,calc,through,backgrounds,matrix,decorations.pathmorphing,positioning}

\usepackage{soul}
\usepackage{mathtools}

\usepackage{calligra}
\usepackage[mathcal]{euscript}
\DeclareMathAlphabet{\mathcalligra}{T1}{calligra}{m}{n}%\mathcalligra{M}
\DeclareMathAlphabet{\mathpzc}{OT1}{pzc}{m}{it}%\mathpzc{M}

\usepackage{geometry} 

\usepackage{stackrel}
\usepackage{rotating}

\usepackage{wasysym}
\usepackage{paralist}
\usepackage{textcomp}
\usepackage[colorinlistoftodos]{todonotes}
\usepackage{caption}

\usepackage{fancyhdr}
\usepackage[bottom]{footmisc}
\usepackage{url}

\usepackage{mathrsfs} %new script
\usepackage{bm}

\calclayout
\makeatletter
\g@addto@macro{\thm@space@setup}{\thm@headfont{\bf}}

\newtheorem{lem}{Lemma}[section]
\newtheorem{prop}[lem]{Proposition}
\newtheorem{cor}[lem]{Corollary}
\newtheorem{thm}[lem]{Theorem}

\theoremstyle{remark}

\theoremstyle{definition}

\newtheorem{defn}[lem]{Definition}

\numberwithin{equation}{section}

\input{operators}

\newcommand{\Z}{\mathbb{Z}}
\newcommand{\cN}{\mathcal{N}}
\newdir{ >}{{}*!/-5pt/@{>}}
\newdir{> }{{}*!/+5pt/@{>}}
\newdir{ <}{{}*!/-5pt/@{<}}
\newdir{>>> }{{}*!/+10pt/@{>}}

\entrymodifiers={+!!<0pt,\fontdimen22\textfont2>}

\usepackage{dsfont}
\begin{document}
%\linenumbers

\title{Hochschild cohomology of some quantum complete intersections}
\author{Karin Erdmann, Magnus Hellstr{\o}m-Finnsen}
%\author{Petter A.\ Bergh, Karin Erdmann, Magnus Hellstr{\o}m-Finnsen}
%\address{Petter A.\ Bergh \\ Institutt for matematiske fag\\ Norges teknisk-naturvitenskapelige universitet \\ N-7491 Trondheim\\ Norway}
%\email{petter.bergh@math.ntnu.no}
\address{Karin Erdmann \\ Mathematical Institute \\ University of Oxford \\ OX2 6GG Oxford \\ United Kingdom}
\email{karin.erdmann@maths.ox.ac.uk}
\address{Magnus Hellstr{\o}m-Finnsen\\ Institutt for matematiske fag\\ Norges teknisk-naturvitenskapelige universitet  \\ N-7491 Trondheim\\ Norway}
\email{magnus.hellstrom-finnsen@math.ntnu.no}
%\curraddr{...}
\date{\today}
\thanks{}
%\translator{...}
\keywords{Hochschild Cohomology; Quantum complete intersections.}
\subjclass[2010]{Primary 16E40; Secondary 16U80; 16S80; 81R50.} 
\begin{abstract}
We compute  the Hochschild cohomology ring of the algebras
$A= k\langle X, Y\rangle/
(X^a, XY-qYX, Y^a)$
over a field $k$
where $a\geq 2$ and where $q\in k$ is a primitive $a$-th root of unity. We find the
the dimension of $\HH^n(A)$ and show that 
it is independent of $a$. We compute explicitly the ring structure
of the even part of the Hochschild cohomology modulo homogeneous nilpotent elements.
\end{abstract}

\maketitle
%\tableofcontents

%%%%%%%%%%%%%%%%%%%%%%%%%%%%%%%%
%%% Introduction
%%%%%%%%%%%%%%%%%%%%%%%%%%%%%%%%

\section{Introduction}\label{sec:intro}
%This project is a continuation of some earlier work of Petter A.\ Bergh and Karin Erdmann. The objective here is to describe the Hochschild cohomology for some quantum complete intersections. 

Let $k$ be a field, and let $0\neq q\in k$. Quantum complete intersections originate from work of Manin \cite{man-87}.  Here we focus on the algebras
$$A_q = k\langle X, Y\rangle /(X^a, XY-qYX, Y^a).
$$
Such algebras have provided
several examples giving answers to homological conjectures and questions. 
Perhaps most spectacular amongst these is Happel's question. 
In \cite{hap-89} 
Happel asked whether an algebra whose Hochschild cohomology is finite-dimensional, must have finite global dimension. 
The main result of  \cite{bgms-05} gave a negative answer: It shows that the Hochschild cohomology of the quantum complete intersection $A_q$
as above, when $a=2$ and $q$ not a root of unity, is finite-dimensional. However
the algebra $A_q$ is selfinjective, hence has infinite global dimension.
Already earlier, R.\ Schulz discovered unusual properties for these algebras $A_q$, see \cite{sch-86} and \cite{sch-94}.

Furthermore,  there is a theory of support varieties in terms of Hochschild cohomology provided the algebra
satisfies suitable finite generation properties, known as condition (Fg) (see
\cite{ehsst-04} and \cite{sol-06}).  
For $A_q$, this condition is satisfied precisely when $q$ is a root of unity.
The general theory of these support varieties
has now been well established in several papers. However, in order to actually compute the varieties over a given
algebra, one needs to determine the ring structure of the Hochschild cohomology, or at least modulo homogeneous nilpotent
elements. 

The results in this paper will be a contribution towards this goal.
We determine the ring structure of the even part of $\HH^{2*}(A)$ modulo the ideal of homogeneous nilpotent
elements  for $A_q$ when $q$ is a primitive $a$-th root of unitity.
The proofs are quite technical, but this illustrates the typical difficulties and computations one is faced
with when trying to compute Hochschild cohomology.

First we present an unpublished result by P.\ Bergh and K.\ Erdmann which determines the dimensions of the Hochschild cohomology groups; this is done via  exploiting Hochschild homology. Surprisingly, the answer is independent of $a$ (see Theorem \ref{thm:dim} and Corollary \ref{cor:dim}). This suggests that perhaps also the ring structure might not depend too much on the parameter $a$. We determine explicit bases of the even part $\HH^{2*}(A)$ (see Section \ref{subsec:basis_ageq3}). 

Furthermore, we compute the algebra structure of  $\HH^{2*}(A)$ modulo
the largest homogeneous nilpotent ideal. We show that it is $\Z_2$-graded, with
degree zero part  isomorphic to the polynomial ring in two variables, generated in degree $2$. The explicit description 
is given in  \ref{R:a=2} when $a=2$, and in \ref{subsec:descoflargeyoneda} when $a\geq 3$.
 
An explicit description when $a=2$ was also given 
 in \cite[Section 3.4]{bgms-05}. We include this case (in Section \ref{sec:HHa2}), as it shows that it is part of the general pattern.

%%%%%%%%%%%%%%%%%%%%%%%%%%%%%%%%%%%%%%%%%%%%%%%
%Oppermann gave also a description of the Hochschild cohomology and homology of general quantum complete intersections in \cite{opp-10}. 
%
%So, the works above clearly give some motivation from the homological point of view in understanding the Hochschild cohomology of quantum complete intersections when they have this nice homological property of being finite dimensional. In this project we will however focus more on the ``algebraic structure'' of Hochschild cohomology. More concretely we will describe the Yoneda product in the cohomology ring by liftings along the minimal projective resolution we are up to describe. 
%
\section{Preliminaries}\label{sec:prelim}

More generally, let $A$ be any finite-dimensional algebra over a field $k$, and let $A^{e}=A \ot_{k} A^{\op}$ denote the \emph{enveloping algebra}.
% where we usually have denoted the opposite multiplication by $*$. 
We view bimodules over $A$ as left  modules over $A^e$. 
In this setting, the  \emph{Hochschild cohomology} of $A$ 
can be taken as  $\HH^n(A)=\Ext^{n}_{A^{e}}(A,A)$,  the $n$-th cohomology of the complex $\Hom_{A^{e}}(\mathbb{P}_{A},A)$, i.e.\
\begin{align}
\Ext^{n}_{A^{e}}(A,A)={\ker d^*_{n+1}}/{\im d^*_n},
\end{align}
where $d^*_n=\Hom_{A^e}(d_n,A)$ and where $d_n$ are the maps in a minimal projective resolution: 
\begin{align}\label{eq:res}
\mathbb{P}:\cdots \to P_2 \xrightarrow{d_2} P_1 \xrightarrow{d_1} P_0 \xrightarrow{\mu} A \to 0.
\end{align}
Then the \emph{Hochschild cohomology } 
\begin{align}
\HH^*(A)=\Ext^*_{A^e}(A,A)
\end{align}
is a $k$-algebra which is graded-commutative.
There are various equivalent  ways to define
the product; here 
we will work with the Yoneda product. 

%In this project we again let $k$ be a field and 

We specialize now to the quantum complete intersections. 
Let  $a$ be an integer such that $a\geq2$. We also let $q \in k$ be a primitive 
$a$th root of unity, and $A$ is the $k$-algebra
defined by 
\begin{align}
A=k\langle X,Y \rangle / (X^a,XY-qYX,Y^a). 
\end{align}
We write $x$ and $y$ for the residue classes of $X$ and $Y$, respectively. 

In \cite{be-08}, for arbitrary parameter $q\neq 0$, 
an explicit  minimal projective bimodule resolution $\mathbb{P}$ as in (\ref{eq:res}) 
%\begin{align}
%\mathbb{P}: \cdots \to P_2 \xrightarrow{d_2} P_1 \xrightarrow{d_1} P_0 \xrightarrow{\mu} A \to 0
%\end{align}
was constructed. 
%(see Equation \ref{eq:res}). 
The $n$th bimodule in $\mathbb{P}$ is
\begin{align} 
P_n=\bigoplus^n_{i=0}A^ef^n_i, 
\end{align}
the free $A^e$-module of rank $n+1$ having generators $\{f^0_n, f^1_n, . . . , f^n_n\}$. For each $s \geq 0$ define the following four elements of $A^e$:
\begin{align}
\tau_1(s)  &= q^s(1\ot x)-   (x\ot1)\label{eq:tau1}\\
\tau_2(s)  &=    (1\ot y)-q^s(y\ot1)\label{eq:tau2}\\
\gamma_1(s)&= \sum_{j=0}^{a-1}q^{js}(x^{a-1-j}\ot x^j      )\label{eq:gamma1}\\
\gamma_2(s)&= \sum_{j=0}^{a-1}q^{js}(y^j      \ot y^{a-1-j})\label{eq:gamma2}
\end{align}
The maps $d_n:P_n \to P_{n-1}$ in $\mathbb{P}$ are given by
\begin{align}
d_{2t}  :f^{2t  }_i&\mapsto 
\begin{cases}\label{eq:diffeven}
\gamma_2\left(\frac{ai}{2}\right)f^{2t-1}_{i}+\gamma_1\left(\frac{2at-ai}{2}\right)f^{2t-1}_{i-1}&\text{for $i$ even}\\
-\tau_2\left(\frac{ai-a+2}{2}\right)f^{2t-1}_{i}+\tau_1\left(\frac{2at-ai-a+2}{2}\right)f^{2t-1}_{i-1}&\text{for $i$ odd}
\end{cases}\\
d_{2t+1}:f^{2t+1}_i&\mapsto
\begin{cases}\label{eq:diffodd}
\tau_2\left(\frac{ai}{2}\right)f^{2t}_{i}+\gamma_1\left(\frac{2at-ai+2}{2}\right)f^{2t}_{i-1}&\text{for $i$ even}\\
-\gamma_2\left(\frac{ai-a+2}{2}\right)f^{2t}_{i}+\tau_1\left(\frac{2at-ai+a}{2}\right)f^{2t}_{i-1}&\text{for $i$ odd}
\end{cases}
\end{align}
where the convention $f^n_{-1} = f^n_{n+1} = 0$ has been used. 
So far, $q$ is arbitrary. Later in our setting we will simplify these 
expressions. %(in Section \ref{sec:} and Section \ref{sec_}). 

%We will first recall a unpublished result by Petter A.\ Bergh and Karin Erdmann before we will describe the Hochschild cohomology and the Hochschild cohomology ring for $A$. 

We will wish to identify nilpotent elements of Hochschild cohomology. This
can be done by exploiting the following result of N. Snashall and \O. Solberg,
see Proposition 4.4 in \cite{ss-04}. 

\begin{prop}\label{nilpotent}  Assume $k$ is a field and $A$ is a finite-dimensional $k$-algebra.
Suppose $\eta$ is a map into $A$ representing  an element of $\HH^n(A)$. If 
${\rm im}(\eta)$ is in the radical of $A$ then $\eta$ is nilpotent in $\HH^*(A)$.\end{prop} 

%%%%%%%%%%%%%%%%%%%%%%%%%%%%%%%%
%%% Dimensions of Hochschild cohomology groups 
%%%%%%%%%%%%%%%%%%%%%%%%%%%%%%%%

\section{Dimensions of Hochschild cohomology groups} \label{sec:dimHH}
We recall an unpublished result by Petter A.\ Bergh and Karin Erdmann which determines the dimensions. 

By viewing $A$ as a left $A^e$-module, it follows from \cite[VI.5.3]{ce-99} that $D(\HH^*(A,A))$ is isomorphic to $\Tor^{A^e}_*(D(A),A)$ as a vector space, where $D$ denotes the usual $k$-dual i.e.\ $D(-):=\Hom_k(-,k)$. In particular, we see that $\dim \HH^n(A) = \dim \Tor^{A^e}_n(D(A),A)$ for all $n\geq0$. Moreover, it follows from \cite{be-08} that $A$ is a Frobenius algebra with Nakayama automorphism $\nu:A \to A$ defined by
\begin{align}
\nu:
\begin{cases}
x \mapsto q^{1-a}x \\
y \mapsto q^{a-1}y.
\end{cases}
\end{align}
The bimodules $D(A)$ and $_\nu A_1$ are isomorphic; 
here the left action on $_{\nu}A_1$ is taken as $a \cdot m := \nu(a)m$. Consequently the dimensions of the Hochschild cohomology of $A$ are given by
\begin{align}
\dim \HH^n(A) = \dim \Tor^{A^e}_n(_\nu A_1,A)
\end{align}
for all $n \geq 0$. 

To compute $\Tor^{A^e}_n(_\nu A_1 , A )$, we tensor the deleted projective bimodule resolution $\mathbb{P}_A$ with the right $A^e$-module $_\nu A_1$. We then obtain an isomorphism 
\begin{center}
\begin{tikzpicture}
\matrix(m)[matrix of math nodes,row sep=2.0em,column sep=3.0em,text height=1.5ex,text depth=0.25ex]
{
   \cdots    & _\nu A_1 \otimes_{A^e}P_{n+1}             &  _\nu A_1 \otimes_{A^e}P_{n  }          & _\nu A_1 \otimes_{A^e}P_{n-1}          &  \cdots          \\
   \cdots    & \oplus^{n+1}_{i=0}(_\nu A_1)e^{n+1}_i     &  \oplus^{n+1}_{i=0}(_\nu A_1)e^{n  }_i  & \oplus^{n+1}_{i=0}(_\nu A_1)e^{n-1}_i  &  \cdots          \\
};
\draw[-> ,         font=\scriptsize](m-1-1) edge         node[above]{$                $} (m-1-2);
\draw[-> ,         font=\scriptsize](m-1-2) edge         node[above]{$ 1 \ot d_{n+1}  $} (m-1-3);
\draw[-> ,         font=\scriptsize](m-1-3) edge         node[above]{$ 1 \ot d_{n  }  $} (m-1-4);
\draw[-> ,         font=\scriptsize](m-1-4) edge         node[above]{$                $} (m-1-5);
\draw[-> ,         font=\scriptsize](m-2-1) edge         node[above]{$                $} (m-2-2);
\draw[-> ,         font=\scriptsize](m-2-2) edge         node[above]{$ \delta_{n+1}   $} (m-2-3);
\draw[-> ,         font=\scriptsize](m-2-3) edge         node[above]{$ \delta_{n  }   $} (m-2-4);
\draw[-> ,         font=\scriptsize](m-2-4) edge         node[above]{$                $} (m-2-5);
\draw[-> ,         font=\scriptsize](m-1-2) edge         node[right]{$ \cong          $} (m-2-2);
\draw[-> ,         font=\scriptsize](m-1-3) edge         node[right]{$ \cong          $} (m-2-3);
\draw[-> ,         font=\scriptsize](m-1-4) edge         node[right]{$ \cong          $} (m-2-4);
\end{tikzpicture}
\end{center}
of complexes, where $\{e^n_0 , e^n_1 , \dots, e^n_n \}$ is the standard generating set of $n + 1$ copies of $_\nu A_1$. Now given an element $\alpha \in k$ and a positive integer $t$, define an element $K_t(\alpha)\in k$ by
\begin{align}
K_t(\alpha) := \sum_{j=0}^{t-1}\alpha^j. 
\end{align}
The map $\delta_n$ is then given by 
\begin{align}
&\delta_{2t}:y^ux^ve^{2t}_i     \mapsto \nonumber\\
&\begin{cases}
qK_a(q^{v+1})y^{u+a-1}x^ve^{2t-1}_i+K_a(q^{u+1})y^ux^{v+a-1}e^{2t-1}_{i-1}&\text{for $i$ even}\\
[q^{v+1}-q^{a-1}]y^{u+1}x^ve^{2t-1}_i+[q^{u+2}-1]y^ux^{v+1}e^{2t-1}_{i-1} &\text{for $i$ odd}
\end{cases} \\
&\delta_{2t+1}:y^ux^ve^{2t+1}_i \mapsto \nonumber\\
&\begin{cases}
[q^{a-1}-q^v]y^{u+1}x^ve^{2t}_i+K_a(q^{u+2})y^ux^{v+a-1}e^{2t}_{i-1}&\text{for $i$ even}\\
qK_a(q^{v+2})y^{u+a-1}x^ve^{2t}_i+[q^{u+1}-1]y^ux^{v+1}e^{2t}_{i-1} &\text{for $i$ odd}
\end{cases}
\end{align}
where we use the convention $e^n_{-1} = e^{n}_{n+1} = 0$. 
This was proved in  \cite{be-08} in a more general setting, and 
by specializing $q$ and using that $x, y$ have the same
nilpotency index, we obtain precisely the above formulae.

For the following result we use this complex to compute the Hochschild cohomology of our algebra $A$, in the case when $q$ is a primitive $a$th root of unity. The result shows that the dimensions of the cohomology groups do not depend on the characteristic of the field.
\begin{thm}\label{thm:dim}
If $q$ is a primitive $a$th root of unity, then $\dim_k \HH^n(A) = 2n + 2$ for all $n\geq0$.
\end{thm}
\begin{proof}
Since $\HH^0(A)$ is isomorphic to the centre of $A$, we see immediately that $\HH^0(A)$ is $2$-dimensional. To find the dimension of $\HH^n(A)$ for $n > 0$, we compute $\ker \delta_{2t}$ for $t\geq1$ and $\ker \delta_{2t+1}$ for $t \geq 0$.

We first compute $\ker \delta_{2t}$ for $t \geq 1$. 
Since $k$ contains a primitive $q$-th root of unity, the characteristic of $k$ does not divide $a$. The equalities $0 = 1 - (q^m)^a = (1 - q^m)K_a(q^m)$, valid for any integer $m$, imply that $K_a(q^v+1) = 0$ if and only if $0 \geq v \geq a-2$, whereas $K_a(q^u+1) = 0$ if and only if $0 \geq u \geq a-2$. Therefore 
\begin{align}
\delta_{2t}(y^ux^ve^{2t}_i) \Leftrightarrow 
\begin{cases}
u\in\{0,\dots,a-2\},\, v\in\{0,\dots,a-2\},\,\text{$i$ even} \\
u=a-1\, v\in\{0,\dots,a-1\},\,\text{$i$ even} \\
u\in\{0,\dots,a-2\},\, v=a-1,\,\text{$i$ even} \\
u=0,\, v=a-1,\,i=2t \\
u=a-1,\, v=0,\,t=0 \\
u=a-2,\, v=a-2,\,\text{$i$ odd} \\
u=a-1,\, v=a-1,\,\text{$i$ odd} \\
\end{cases}
\end{align}
and there are $a^2 t + a^2$ such elements. As for linear combinations in $\ker \delta_{2t}$ of at least two basis vectors from $\oplus^{2t}_{i=0}(_\nu A_1)e^{2t}_i$, they appear in pairs. Namely, they are linear combinations of the pairs
\begin{align}
x^{a-1}e^{2t}_i +C y^{a-1}x^{a-2}e^{2t}_{i+1}   &i=0,2,\dots,2t-2 \\ 
y^{a-1}e^{2t}_i +C'y^{a-2}x^{a-1}e^{2t}_{i-1}   &i=2,4,\dots,2t 
\end{align}
where $C$ and $C'$ are scalars whose values depend on all the parameters involved. There are $2t$ such elements, hence $\dim_k \ker \delta_{2t} = (a^2 + 2)t + a^2$.

Next we compute $\ker \delta_{2t+1}$ for $t \geq 0$, recall that 
the characteristic of $k$ does not divide $a$. We see that  
\begin{align}
\delta_{2t+1}(y^ux^ve^{2t+1}_i) = 0 \Leftrightarrow 
\begin{cases}
u=a-1\, v\in\{0,\dots,a-1\},\,\text{$i$ arbitrary} \\
u\in\{0,\dots,a-2\},\, v = a - 1,\,\text{$i$ arbitrary} 
\end{cases}
\end{align}
and there are $(2a-1)(2t+2)$ such elements. As for linear combinations in $\ker \delta_{2t+1}$ of at least two basis vectors from $\oplus^{2t+1}_{i=0}(_\nu A_1)e^{2t+1}_i$, also here they appear in pairs. Namely, they are linear combinations of the pairs 
\begin{align}
y^ux^ve^{2t+1}_i + Cy^{u+1}x^{v-1}e^{2t+1}_{i+1}
\end{align}
for $0 \leq u \leq a-3$, $1 \leq v \leq a-2$ and $i=0,2,\dots,2t$, and where $C$ is a scalar. There are $(a - 2)^2(t + 1)$ such pairs (note that, if $a = 2$, then the requirements on $u$ and $v$ are empty, but this causes no problem since by the formula there are no pairs in this situation), and so $\dim_k \ker \delta_{2t+1} = (a^2 + 2)t + a^2 + 2$.

We have now computed $\ker \delta_{2t}$ for $t \geq 1$ and $\ker \delta_{2t+1}$ for $t \geq$. Using the equalities
\begin{align}
\dim_k\im\delta_n+\dim_k\ker\delta_n=\dim_k\oplus_{i=0}^n(_\nu A_1)e^n_i)=(n+1)a^2, 
\end{align}
we see that $\dim_k \im \delta_{2t+1} = \dim_k \im \delta_{2t+2} = (a^2 - 2)(t + 1)$. Consequently
\begin{align}
\dim_k \HH^{2t+1}(A) &= \dim_k \ker \delta_{2t+1} - \dim_k \im \delta_{2t+2} \\
                     &= 4t+4 \\
\dim_k \HH^{2t+2}(A) &= \dim_k \ker \delta_{2t+2} - \dim_k \im \delta_{2t+3} \\ 
                     &= 4t+6
\end{align}
for $t \geq 0$, and the proof is complete. 
\end{proof}

This result implies immediately the following: 

\begin{cor}\label{cor:dim}
The dimension of the cohomology groups $\HH^n(A)$ is independent of $a$. 
\end{cor}

%With this result in mind we have motivation for finding general expressions of the cohomology groups. 

%%%%%%%%%%%%%%%%%%%%%%%%%%%%%%%%
%%% Hochschild cohomology when $a=2$
%%%%%%%%%%%%%%%%%%%%%%%%%%%%%%%%

\section{Hochschild cohomology when $a=2$} \label{sec:HHa2}
In this section we let $a=2$ and $q=-1$ (and $\mathrm{char}(k) \neq 2$), so we have that 
\begin{align}
A=k\langle X,Y \rangle/(X^2, XY+YX,Y^2). 
\end{align}
We write $x, y$ again for the images of $X, Y$ in $A$. 
We will rewrite the differentials for the minimal projective resolution
before studying  the even cohomology ring $\HH^{2*}(A)$ for this case. 

%%%%%
\subsection{Minimal projective resolution when $a=2$} 
We introduce the following notation: 
\begin{align}
\beta_y  = (1 \ot y)+(y \ot 1) && \beta_x  = (1 \ot x)+(x \ot 1) \\
\alpha_y = (1 \ot y)-(y \ot 1) && \alpha_x = (1 \ot x)-(x \ot 1)
\end{align}
Now we can rewrite the differentials for the minimal projective resolution $\mathbb{P}$ in Equation \ref{eq:diffeven} and \ref{eq:diffodd}; we get: 
\begin{align}\label{eq:difffin}
d_n(f^n_i)= 
\begin{cases}
(-1)^i(\beta_y f^{n-1}_i + \beta_x f^{n-1}_{i-1}) \qquad \text{when $n$ is even}\\
(-1)^i(\alpha_yf^{n-1}_i - \alpha_xf^{n-1}_{i-1}) \qquad \text{when $n$ is odd }. 
\end{cases}
\end{align}

%%%%%
\subsection{Description of cohomology groups}\label{subsec:cohgrps}
In Section \ref{sec:dimHH} we have seen that $\dim \HH^n(A)= 2n+2$. Knowing this, we will determine a basis for $\HH^n(A)$ for arbitrary even degrees $n$. We write $\delta_{ir}$ as usual for the Kronecker symbol.

%{\tt maybe leave $n$ odd since we don't compute products for odd degrees? }
%
\begin{lem}\label{lem:basisHH} Let $n=2t$. For $r=0, 1, \ldots, 2t$  define maps $\xi_r, \eta_r: P_{2t}\to A$ as follows.
\begin{align}
\xi_r(f_i^{2t}) = \delta_{ir}\cdot 1_A, \qquad \eta_r(f_i^{2t}) =  \delta_{ir}\cdot xy.
\end{align}
(a) \ The classes of these maps form a basis for $\HH^{2t}(A)$.\\
(b) \ The classes of the $\eta_r$ give nilpotent elements in $\HH^*(A)$. 
\end{lem}

\begin{proof} Part (b) will follow from Proposition \ref{nilpotent}. We prove 
now part (a). 
Note that these are $2n+2$ elements, so we only have to show that the maps are in the kernel of $d_{2t+1}^*$, and that they are linearly independent modulo the image of $d_{2t}^*$. 
\begin{itemize}
\item[(1)]  We apply $\xi_r$ to $d_{2t+1}(f_i^{2t+1})$, this gives
\begin{align}
\xi_r[(-1)^i(\alpha_yf_i^{2t} - \alpha_xf_{i-1}^{2t})] =
(-1)^i [\alpha_y[\delta_{ir}\cdot 1_A] - \alpha_x[\delta_{i-1,r}\cdot 1_A] ]= 0
\end{align}
(we view $A$ as a left $A^e$ module, and $\alpha_y \cdot 1_A=0= \alpha_x\cdot 1_A$). Similarly we apply $\eta_r$ to $d_{2t+1}(f_i^{2t+1})$ and get 
\begin{align}
\eta_r[(-1)^i(\alpha_yf_i^{2t} - \alpha_xf_{i-1}^{2t})] =
(-1)^i [\alpha_y[\delta_{ir}\cdot xy] - \alpha_x[\delta_{i-1,r}\cdot xy]] = 0
\end{align} 
(since $xy$ is in the socle of $A$ we see that $\alpha_y\cdot xy=0$ and $\alpha_x\cdot xy=0$).
%{\tt maybe remove the dots in the products?}
\item[(2)]  Let $c_r, d_r\in K$ and $\rho: P_{2t-1}\to A$ such that
\begin{align}
\sum_{r=0}^{2t} c_r\xi_r + d_r\eta_r  = \rho\circ d_{2t} \in {\rm im}(d_{2t}^*).
\end{align}
We must show that  $c_r =0 = d_r$ for all $r$. Write $\rho(f_i^{2t-1}) = p_i = a_i + b_ix + c_iy + d_ixy \in A$. Then we have
\begin{align}
\rho\circ d_{2t}(f_i^{2t}) = (-1)^i[\beta_y p_i + \beta_x p_{i-1}]  = (-1)^i[2a_iy + 2a_{i-1}x]
\end{align}
which are elements in $A$. On the other hand if we apply the map given by the sum to $f_i^{2t}$ then we get
\begin{align}
c_i \ + \ d_ixy
\end{align}
also elements in $A$. We assume these are equal, and it follows that all scalars are zero.
\end{itemize}
\end{proof}
\subsection{Products in even degrees of $\HH^*(A)$}
%We continue with studying even degrees. 
Recall that the even part $\HH^{2*}(A)$ is a subring of the Hochschild cohomology, and it is commutative. 
The aim of this section is to prove the following:

\medskip

\begin{thm} \label{R:a=2} Let
$k$ be a field with char$(k)\neq 2$, and let
$A=k\langle X, Y\rangle/(X^2, XY+YX, Y^2)$. 
 $R$ be the  subalgebra $R:= k\langle \xi_i^{2t}: t\geq 0, 0\leq i\leq 2t \rangle$ of the even Hochschild cohomology ring of $A$. This is $\mathbb{Z}_2$-graded, with 
\begin{align}
R_0:= k\langle \xi_{i}^{2t}: \text{$i$ even } \rangle, \qquad
R_1:= k\langle \xi_{i}^{2t}: \text{$i$ odd }  \rangle.
\end{align}
Then $R_0$ is isomorphic to the polynomial ring $k[z_0, z_1]$ where we identify $\xi_0^2$ with $z_0$ and $\xi_2^2$ with $z_1$. Moreover, the odd part is equal to $R_1 = R_0\xi_1^2$ and $\xi_1^2\cdot \xi_1^2=\xi_2^4$.
\end{thm}

\begin{cor} Let $\cN$ be the largest homogeneous nilpotent ideal of
$\HH^{2*}(A)$. Then 
\begin{align}
\HH^{2*}(A)/\cN \cong R.
\end{align}
\end{cor}

We fix a degree $2t$, and we will compute the product of a general element $\xi$ of degree $2t$ with an element $\chi$ of degree $2m$ and we let $2m$ vary. We take representatives $\xi:P_{2t} \to A$ and $\chi:P_{2m} \to A$ which 
are $k$-linear combinations of the basis.  Let%We will Yoneda define products of two elements, say $\xi\in\HH^{2t}(A)$ and $\chi\in\HH^s(A)$, as a lifting of $\xi$ along the minimal projective resolution. For this purpose we recall that 
%\begin{align}
%\xi:P_{2t} \to A \qquad \chi:P_{s} \to A
%\end{align}
%are representatives for elements in $\HH^{2t}(A)$ and $\HH^s(A)$, respectively, whose image will be denoted by 
\begin{align}
\xi(f^{2t}_i)=p_i             \in A & \qquad \text{with $0 \leq i \leq 2t$} \\ 
\chi(f^{2m}_i)=\overline{p}_i \in A & \qquad \text{with $0 \leq i \leq 2m $}.  
\end{align}
By (4.5), the elements $p_i$ and $\bar{p}_i$ are then in the centre of $A$, 
we will use this freely.
 
\begin{defn}
The Yoneda product $\xi \bullet \chi$ is the residue class of $\chi \circ h_{2m}$ where the family $(h_s)$ with $h_s:P_{2t+s}\to P_s$ is a lifting of $\xi$. That is, we have the following diagram: 
%Let $\xi\in\HH^{2t}(A)$ and $\chi\in\HH^{s}(A)$, we define the \emph{Yoneda product} $\chi \bullet \xi \in \HH^{2t+s}(A)$ as the residue class of the composition 
%\begin{align}
%\chi \circ h_s : P_{2t+s} \to A. 
%\end{align}
%where $h_s$ is the lifting defined above. 
\end{defn}
%
%
%We illustrate this by lifting $\xi$ (similar procedure could be derived for $\chi$). Consider the following digram, 
\begin{center}
\begin{tikzpicture}
\matrix(m)[matrix of math nodes,row sep=2.0em,column sep=2.3em,text height=1.5ex,text depth=0.25ex]
{
  P_{2t+2m}       & P_{2t+2m-1}     &  \cdots           &  P_{2t+s}       & P_{2t+s-1}     &  \cdots         & P_{2t+1}      &  P_{2t}        &                  \\
  P_{2m}          & P_{2m-1}        &  \cdots           &  P_{s}          & P_{s-1}        &  \cdots         & P_{1}         &  P_0           &   A              \\
  A               &                 &                   &                 &                &                 &               &                &                  \\
};
\draw[-> ,         font=\scriptsize](m-1-1) edge         node[above]{$ d_{2t+2m}      $} (m-1-2);
\draw[-> ,         font=\scriptsize](m-1-2) edge         node[above]{$                $} (m-1-3);
\draw[-> ,         font=\scriptsize](m-1-3) edge         node[above]{$                $} (m-1-4);
\draw[-> ,         font=\scriptsize](m-1-4) edge         node[above]{$ d_{2t+s}       $} (m-1-5);
\draw[-> ,         font=\scriptsize](m-1-5) edge         node[above]{$                $} (m-1-6);
\draw[-> ,         font=\scriptsize](m-1-6) edge         node[above]{$                $} (m-1-7);
\draw[-> ,         font=\scriptsize](m-1-7) edge         node[above]{$ d_{2t+1}       $} (m-1-8);
\draw[-> ,         font=\scriptsize](m-2-1) edge         node[above]{$ d_{2m}         $} (m-2-2);
\draw[-> ,         font=\scriptsize](m-2-2) edge         node[above]{$                $} (m-2-3);
\draw[-> ,         font=\scriptsize](m-2-3) edge         node[above]{$                $} (m-2-4);
\draw[-> ,         font=\scriptsize](m-2-4) edge         node[above]{$ d_{s}          $} (m-2-5);
\draw[-> ,         font=\scriptsize](m-2-5) edge         node[above]{$                $} (m-2-6);
\draw[-> ,         font=\scriptsize](m-2-6) edge         node[above]{$                $} (m-2-7);
\draw[-> ,         font=\scriptsize](m-2-7) edge         node[above]{$ d_1            $} (m-2-8);
\draw[-> ,         font=\scriptsize](m-2-8) edge         node[above]{$ \mu            $} (m-2-9);
\draw[-> ,         font=\scriptsize](m-1-1) edge         node[right]{$ h_{2m}         $} (m-2-1);
\draw[-> ,         font=\scriptsize](m-1-2) edge         node[right]{$ h_{2m-1}       $} (m-2-2);
\draw[-> ,         font=\scriptsize](m-1-4) edge         node[right]{$ h_{s}          $} (m-2-4);
\draw[-> ,         font=\scriptsize](m-1-5) edge         node[right]{$ h_{s-1}        $} (m-2-5);
\draw[-> ,         font=\scriptsize](m-1-7) edge         node[right]{$ h_1            $} (m-2-7);
\draw[-> ,         font=\scriptsize](m-1-8) edge         node[right]{$ h_0            $} (m-2-8);
\draw[-> ,         font=\scriptsize](m-1-8) edge         node[above]{$ \xi            $} (m-2-9);
\draw[-> ,         font=\scriptsize](m-2-1) edge         node[right]{$ \chi           $} (m-3-1);
\end{tikzpicture}
\end{center}
%where the \emph{lifting maps} are given by 
where $\xi=\mu \circ h_0$ and where all squares commute. 
%The lifting maps are given by 
We define  maps $h_s$ ($0 \leq s \leq 2m$), and will show that they are a lifting. 
\begin{align}
h_s(f^{2t+s}_i)=
\begin{cases}
            \sum_{j=0}^s      p_{i-j}f^s_j        & \text{when $s$ even} \\
(-1)^i\left(\sum_{j=0}^s(-1)^jp_{i-j}f^s_j\right) & \text{when $s$ odd }
\end{cases}
\end{align}

\begin{prop}
The  maps $h_s$ for $0 \leq s \leq s$ make the lifting diagram commutative, i.e.\ $d_s \circ h_s = h_{s-1} \circ d_{2t+s}$. 
\end{prop}
\begin{proof}
When $2t$ is fixed the proof of this result is an examination when $s$ is even and when $s$ odd, and the result follows from explicit calculations. \\
{\bf Case $s$ even}: We have 
%Calculations on page 26
\begin{align}
(d_s \circ h_s)(f^{2t+s}_i)&=d_s\circ\left(\sum_{j=0}^sp_{i-j}f^s_j\right)=\sum_{j=0}^sp_{i-j}d_s(f^s_j)\\
                           &=\sum_{j=0}^sp_{i-j}(-1)^j\left(\beta_yf^{s-1}_j+\beta_xf^{s-1}_{j-1}\right)\nonumber\\
(h_{s-1} \circ d_{2t+s})(f^{2t}_i)&=h_{s-1} \circ \left((-1)^i\left(\beta_yf^{2t+s-1}_i+\beta_xf^{2t+s-1}_{i-1}\right)\right)\\
                           &=\beta_yh_{s-1}(f^{2t+s-1}_i)+\beta_xh_{s-1}(f^{2t+s-1}_{i-1})\nonumber\\
                           &=\sum^{s-1}_{j=0}(-1)^j\beta_yp_{i-j}f^{s-1}_j+(-1)^{i-1}\sum^{s}_{j=0}(-1)^j\beta_xp_{i-(j+1)}f^{s-1}_j\nonumber\\
                           &=\sum_{j=0}^s(-1)^jp_{i-j}(\beta_yf^{s-1}_j+\beta_xf^{s-1}_{j-1})\nonumber
\end{align}
We observe that the expressions are equal, hence $d_s \circ h_s=h_{s-1} \circ d_{2t+s}$ and we conclude that $h_s$ is a lifting map when $s$ is even. \\
{\bf Case $s$ odd}: We calculate 
\begin{align}
(d_s \circ h_s)(f^{2t+s}_i)&=d_s\circ\left((-1)^i\sum_{j=0}^s(-1)^jp_{i-j}f^s_j\right)=\sum_{j=0}^s(-1)^jp_{i-j}d_sf^s_j\\
                           &=(-1)^i\sum_{j=0}^s(-1)^jp_{i-j}(-1)^j(\alpha_yf^{s-1}_j-\alpha_xf^{s-1}_{j-1})\nonumber\\
                           &=(-1)^i\sum_{j=0}^sp_{i-j}(\alpha_yf^{s-1}_j-\alpha_xf^{s-1}_{j-1})\nonumber\\
(h_{s-1} \circ d_{2t+s})(f^{2t}_i)&=h_{s-1} \circ \left((-1)^i\left(\alpha_yf^{2t+s-1}_i-\alpha_xf^{2t+s-1}_{i-1}\right)\right)\\
                           &=(-1)^i\sum_{j=0}^s(-1)^jp_{i-j}(\alpha_yf^{s-1}_j-\alpha_xf^{s-1}_{j-1})\nonumber
\end{align}
The expressions are equal, which proves the odd case as well, and that $h_s$ is a lifting map after all. 
\end{proof}

%We investigate further for $s$ even we have 
%\begin{align}
%(\chi \bullet \xi)(f^{2t+s}_i) = ( \chi \circ h_s )(f^{2t+s}_i) = \chi \circ \left( \sum_{j=0}^s p_{i-j}f^s_j \right) = \sum_{j=0}^s p_{i-j}\chi(f^s_j) = \sum_{j=0}^s p_{i-j}\overline{p}_j
%\end{align}
%
%From Section \ref{subsec:cohgrps} we had that $p_i$ and $\overline{p}_i$ take the forms 
%\begin{align}
%\begin{cases}
%p_0=a_0+b_0x+d_0xy \\
%p_i=a_i+d_ixy \qquad \text{for $0<i<2t$ }\\
%p_{2t}=a_{2t}+c_{2t}y+d_{2t}xy
%\end{cases}
%\begin{cases}
%\overline{p}_0   =\overline{a}_0   +\overline{b}_0x   +\overline{d}_0xy \\
%\overline{p}_i   =\overline{a}_i    \overline{ }      +\overline{d}_ixy \qquad \text{for $0<i<s $ }\\
%\overline{p}_{s }=\overline{a}_{s }+\overline{c}_{s }y+\overline{d}_{s }xy
%\end{cases}
%\end{align}
%so we again observe that the exceptions happen in the ``endpoints''.... {\bf NOT WELL-DEFINED} in the endpoints where we have $p_0=a_0+b_0x+d_0xy$ and $p_{2t}=a_{2t}+c_{2t}y+d_{2t}xy$. Similarly $\overline{p}_i=\overline{a}_i+\overline{d}_ixy$ etc. Hence an arbitrary element in the bullet product look like 
%%26'
%\begin{align}
%p_{i-j}\overline{p}_j&=(a_{i-j}+b_{i-j}x+c_{i-j}y+d_0xy)(\overline{a}_j+\overline{b}_jx+\overline{c}_jy+\overline{d}_jxy)\\
%                     &=
%\end{align}

\subsection{Description of Yoneda products}
In Section \ref{subsec:cohgrps} we have described a basis for $\HH^{2t+2m}(A)$. Now we compute the Yoneda product of $\xi\in\HH^{2t}(A)$ and $\chi\in\HH^{2m}(A)$ as linear combination of $\kappa_r$ and $\eta_r$ for $0 \leq r \leq 2t+2m$:  
\begin{align}
\chi \bullet \xi = \sum^{2t + 2m}_{r=0} u_r \kappa_r + v_r \eta_r 
\end{align}
where 
\begin{align}
\kappa_r(f^{2t+2m}_i)=\delta_{ir} \cdot 1_A \qquad \eta_r(f^{2t+s}_i)=\delta_{ir} \cdot xy 
\end{align}
hence we get 
\begin{align}
( \xi \bullet \chi )(f^{2t+2m}_i) = u_i \cdot 1_A + v_i \cdot xy.  
\end{align}
By letting $u_i$ and $v_i$ denote maps 
\begin{align}
u_i: f^{2t+2m}_i \mapsto x \qquad \text{and} \qquad v_i: f^{2t+s}_i \mapsto y
\end{align}
we observe that this is isomorphic to the polynomial algebra with two generators. 
%Moreover we have formulated 
%\begin{align}
%( \xi \bullet \chi )(f^{2t+s}_i) = (\chi \circ h_s)(f^{2t+s}_i) = \chi \circ \sum^s_{j=0} p_{i-j} f^s_j = \sum^s_{j=0} p_{i-j} \overline{p}_j
%\end{align}
%(where as usual $\xi(f^{2t}_i)=p_i$ and $\chi(f^s_i)=\overline{p}_i$). 

\begin{cor} Let $\xi(f_r^{2t}) = p_r\in A$ and $\chi(f_r^{2m}) = \bar{p}_r\in A$. Then
\begin{align}
\chi\circ h_{2m}(f_i^{2t+2m}) = \sum_{0\leq j\leq 2m \text{ and } 0\leq i-j\leq 2t} p_{i-j}\bar{p}_j. 
\end{align}
%where the sum is over $0\leq j\leq 2m$ and $0\leq i-j\leq 2t$. 
In particular if we let $\xi_i^{2m}$ and $\xi_j^{2t}$ denote the basis elements of Lemma \ref{lem:basisHH} then we have
\begin{align}
\xi_i^{2m}\cdot \xi_j^{2t} = \xi_{i+j}^{2t+2m}.
\end{align}
\end{cor}
\begin{proof}
We apply the lifting formula and obtain the first part directly. If we take $\chi = \xi_{i}^{2m}$ and $\xi = \xi_j^{2t}$ then $p_i=1$ and $p_r=0$ for $r\neq i$ and similarly $\bar{p}_j=1$ and $\bar{p}_r=0$ otherwise. So we get that the image of $f_{\nu}^{2t+2m}$ is $1$ if $\nu = i+j$ and is zero otherwise. The last part follows.
\end{proof}

\subsection{Completing  the proof of Theorem \ref{R:a=2}.} \ 
To see that $R_0$ is as stated, we let the  $\xi_{0}^{2t}$ correspond to the powers of $z_0$, and we let the $\xi_{2t}^{2t}$ correspond to the powers of $z_1$ and in general $\xi_i^{2t}$ for $i$ even  corresponds to the product $z_0^{2t-i}z_1^{i}$. The formula in the Corollary may be used to see that this gives an algebra isomorphism. The rest follows from the previous results.

The Corollary is a direct consequence: The intersection of $R$ with
$\cN$ is zero, and as we have observed, any element in the span of
maps $\eta_r$ is in $\cN$. 
$\Box$

%%%%%%%%%%%%%%%%%%%%%%%%%%%%%%%%
%%% Cohomology for general $a$-cases 
%%%%%%%%%%%%%%%%%%%%%%%%%%%%%%%%

\section{Cohomology for $a\geq 3$} 
Now we study the cohomology when $a \geq 3$. Still let $q$ be an $a$th root of unity and assume the algebra is
\begin{align}
A={k\langle X,Y \rangle}/{(X^a,XY-qYX,Y^a)}.
\end{align}
We write again $x, y$ for the images of $X, Y$ in $A$.

%\subsection{Restrictions and preliminaries} 
%It seems necessary to restrict to cases where elements $p \in A$ commutes with 

\subsection{Differentials} 
We assume $a \geq 3$, then we can simplify the differentials defined in \ref{eq:diffeven} and \ref{eq:diffodd}. We observe that the elements in $A$ introduced in \ref{eq:tau1} to \ref{eq:gamma2} depend only on the parity of $s$ modulo $a$, and the arguments in \ref{eq:diffeven} and \ref{eq:diffodd} make only use of the cases where $s \equiv 0$ or $s \equiv 1$ modulo $a$. Using this the differentials take the following form which we will use from now: %%%%%%%%%%%%%%%
%We recall the differentials $d_n:P_n \to P_{n-1}$ in Equation 
%\ref{eq:diffeven} and \ref{eq:diffodd} are defined as 
%\begin{align}
%d_{2t}  :f^{2t  }_i&\mapsto 
%\begin{cases}
%\gamma_2\left(\frac{ai}{2}\right)f^{2t-1}_{i}+\gamma_1\left(\frac{2at-ai}{2}\right)f^{2t-1}_{i-1}&\text{for $i$ even}\\
%-\tau_2\left(\frac{ai-a+2}{2}\right)f^{2t-1}_{i}+\tau_1\left(\frac{2at-ai-a+2}{2}\right)f^{2t-1}_{i-1}&\text{for $i$ odd}
%\end{cases}\\
%d_{2t+1}:f^{2t+1}_i&\mapsto
%\begin{cases}
%\tau_2\left(\frac{ai}{2}\right)f^{2t}_{i}+\gamma_1\left(\frac{2at-ai+2}{2}\right)f^{2t}_{i-1}&\text{for $i$ even}\\
%-\gamma_2\left(\frac{ai-a+2}{2}\right)f^{2t}_{i}+\tau_1\left(\frac{2at-ai+a}{2}\right)f^{2t}_{i-1}&\text{for $i$ odd}
%\end{cases}
%\end{align}
%We recall also that 
%\begin{align}
%\tau_1(s)  &= q^s(1\ot x)-   (x\ot1)\\
%\tau_2(s)  &=    (1\ot y)-q^s(y\ot1)\\
%\gamma_1(s)&= \sum_{j=0}^{a-1}q^{js}(x^{a-1-j}\ot x^j      )=(x\ot1)+q^s(1\ot x)\\
%\gamma_2(s)&= \sum_{j=0}^{a-1}q^{js}(y^j      \ot y^{a-1-j})=(y\ot1)+q^s(1\ot y)
%\end{align}
%Since the arguments $s$ of $\gamma_j(s)$ and $\tau_j(s)$ (for $j\in\{1,2\}$) only effect the power of $q$, hence they are only dependent on the value $s$ modulo $a$. We will only write $s$ for this and should be understand as $ s\mod a $. The differentials simplifies to 
\begin{align}
d_{2t}  :f^{2t  }_i&\mapsto 
\begin{cases}
 \gamma_y(0)f^{2t-1}_{i} + \gamma_x(0)f^{2t-1}_{i-1}&\text{for $i$ even}\\
-\tau_y  (1)f^{2t-1}_{i} + \tau_x  (1)f^{2t-1}_{i-1}&\text{for $i$ odd}
\end{cases}\\
d_{2t+1}:f^{2t+1}_i&\mapsto
\begin{cases}
 \tau_y  (0)f^{2t}_{i} + \gamma_x(1)f^{2t}_{i-1}&\text{for $i$ even}\\
-\gamma_y(1)f^{2t}_{i} + \tau_x  (0)f^{2t}_{i-1}&\text{for $i$ odd}
\end{cases}
\end{align}
where we have replaced 
\begin{align}
\tau_1 = \tau_x \qquad \tau_2 = \tau_y \qquad \gamma_1 = \gamma_x \qquad \gamma_2 = \gamma_x 
\end{align}

\subsection{A basis for $\HH^{2t}(A)$  for $a\geq 3$}\label{subsec:basis_ageq3}
As observed the dimension of the degree $2t$ part is always $4t+2$ which is independent of $a$. We therefore expect that there should be a basis when $a\geq 3$ which is not so different from the one we had for  $a=2$.

\begin{defn}\label{def:basisgeneral}
Let $\zeta_j: P_{2t}\to A$ be the map 
\begin{align}
\zeta_j(f_i^{2t}) = 
\begin{cases}
1 & i=j \\
0 & \text{else.}
\end{cases}
%\left\{\begin{array}{ll} 1 & \ i=j\cr 0 & \ \mbox{ else \ }
%\end{array}
%\right.
\end{align}
Let $j$ be even, then define 
\begin{align}
\eta^+_j(f_i^{2t}) = 
\begin{cases}
x^{a-1}y^{a-1} & i=j \\
0              & \text{else.}
\end{cases}
%\left\{\begin{array}{ll} x^{a-1}y^{a-1} & i=j\cr 0 & \mbox{ else }
%\end{array}
%\right.
\end{align}
Now let $j$ be odd, then define
\begin{align}
\eta^-_j(f_i^{2t}) = 
\begin{cases}
xy & i=j\\ 
0  & \text{ else. }
\end{cases}
%\left\{\begin{array}{ll} xy & i=j\cr 0 & \mbox{ else }
%\end{array}
%\right.
\end{align}
\end{defn}
%%%%%%%%%%%%%%%%
\begin{lem} We fix a degree $2t$.\\ 
(a) \ The classes of the elements $\zeta_i$ and $\eta_j^{\pm}$ as defined above form a basis of $\HH^{2t}(A)$.\\
(b) \ The classes of the elements $\eta_j^{\pm}$ give nilpotent elements of
 $\HH^*(A)$.
\end{lem}
\begin{proof} Part (b) will follow again from Proposition \ref{nilpotent}. We prove
now part (a). 
These are in total $4t+2$ maps, so we only 
have to show that 
the maps lie in the kernel of $d_{2t+1}^*$, and 
that they are linearly independent modulo the image of $d_{2t}^*$.

\noindent (1) \ 
Let $\xi$ be one these maps. We write $\xi(f_i^{2t}) = p_i \in A$, so that $p_i$ is either $0$ or $1$ or one of $x^{a-1}y^{a-1}$ or $xy$ depending on the parity of $i$. We need to check that $\xi(d_{2t+1}(f_i^{2t}))=0$. 
\begin{itemize}
\item[(a)] Assume $i$ is even, then this is equal to
\begin{align}
\xi(\tau_y(0)f_i^{2t} + \gamma_x(1)f_{i-1}^{2t}) = \tau_y(0)p_i + \gamma_x(1)p_{i-1}.
\end{align}
This has to be calculated in $A$ which is viewed as an $A^e$ left module. We have 
\begin{align}
\tau_y(0)p_i = p_iy - yp_i
\end{align}
This is zero if $p_i=1$. Otherwise since $i$ is even we only need to consider $p_i=x^{a-1}y^{a-1}$ and then $p_iy=0$ and $yp_i=0$. Next, if $p_{i-1}=1$ then 
\begin{align}
\gamma_x(1)p_{i-1} = \sum_{j=0}^{a-1} q^jx^{a-1-j}\cdot 1 \cdot x^{j} = (\sum_{j=0}^{a-1}q^j) x^{a-1}
\end{align}
and this is zero, note that  $1 + q + \ldots + q^{a-1} = 0$ since $q$ is an $a$-th root of $1$. Otherwise $p_{i-1}= xy$ and then 
\begin{align}
\gamma_x(1)p_{i-1} = \sum_{j=0}^{a-1} q^j(x^{a-1-j}xyx^j)
\end{align}
and this is a scalar multiple of $x^ay$ and hence is zero. 
\item[(b)] Let $i$ be odd, we get
\begin{align}
\xi(-\gamma_y(1)f_i^{2t} + \tau_x(0)f_{i-1}^{2t}) = -\gamma_y(1)p_i + \tau_x(0)p_{i-1}.
\end{align}
By calculations similar to part (a) we see that this is zero in all cases to be considered.
\end{itemize}

\noindent (2) \ 
 We consider a linear combination of the above elements and assume that it lies in the image of $d_{2t}^*$. Explicitly let 
\begin{align}
\sum_{j=0}^{2t} c_j\zeta_j + \sum_{j \text{ even} } s_j^+ \eta_j^+ + \sum_{j \text{ odd}} s_j^-\eta_j^- = \xi\circ d_{2t}
\end{align}
where $\xi: P_{2t-1}\to A$, with $c_j$ and $s_j^{\pm}$ in $k$. We must show that this is only possible, as $\xi$ varies, with all $c_j$ and $s_{j}^{\pm}$ equal to zero.
\begin{itemize}
\item[(a)] Apply the LHS to $f_i^{2t}$ with $i$ even, this gives 
\begin{align}
c_i + s_i^+(x^{a-1}y^{a-1}).
\end{align}
On the other hand,
\begin{align}
\xi\circ d_{2t}(f_i^{2t}) = \gamma_y(0)\xi(f_i^{2t-1}) + \gamma_x(0)\xi(f_{i-1}^{2t-1}).\label{eq:star}%\leqno{(*)}
\end{align}
This is an element in $A$ viewed as an $A^e$ left module. For any element $z\in A$, $\gamma_x(0)z$ or $\gamma_y(0)z$ can never have a constant term since $\gamma_x(0)$ and $\gamma_y(0)$ are in the radical of $A^e$. Hence the Equation (\ref{eq:star}) does never have a non-zero constant term and $c_i=0$. 

We claim that we also cannot get a term which is a multiple of $x^{a-1}y^{a-1}$. Namely if so this could only come from either $\gamma_y(0) x^{a-1}$ or from $\gamma_x(0)y^{a-1}$. Now, 
\begin{align}
\gamma_y(0)x^{a-1} = \sum_{j=0}^{a-1} y^jx^{a-1}y^{a-1-j} = \sum_{j=0}^{a-1} (q^{-1})^{j(a-1)} x^{a-1}y^{a-1} = 0
\end{align}
since $\sum_{j=1}^{a-1}q^j = 0$. Hence $s_i^+=0$. 
\item[(b)] Apply the LHS to  $f_i^{2t}$ with $i$ odd, this gives
\begin{align}
c_i + s_i^-(xy).
\end{align}
On the other hand, 
\begin{align}
\xi\circ d_{2t}(f_i^{2t}) = -\tau_y(1)\xi(f_i^{2t-1}) + \tau_x(1)\xi(f_{i-1}^{2t-1}).\label{eq:doublestar}%\leqno{(**)}
\end{align}
As before, since $\tau_y(1)$ and $\tau_x(1)$ are in the radical of $A^e$, this cannot have non-zero constant terms. Hence $c_i=0$.

We must check that we cannot get $xy$. If $xy$ should occur in $\tau_y(1)$ this can only come from $\tau_y(1)x$ but this is equal to $xy-qyx = 0$. Similarly $\tau_x(1)y=0$ and we do not get $xy$. Hence $s_i^-=0$. 
\end{itemize}
We have proved that the $4t+2$ maps are linearly independent modulo the image of $d_{2t}^*$. By dimensions, they are a basis of $\HH^{2t}(A)$.
\end{proof}

The aim of this section is to prove the following.

\begin{thm}\label{R:a>2} Let $k$ be a field, $a \geq 3$ an integer, $q \in k$ a primitive $a$th root of unity, and $A$ the quantum complete intersection $k\langle X,Y \rangle/(X^a,XY-qYX,Y^a)$. Assume 
$$R = k-{\rm Sp} \{ \zeta_i^{2t}: \ t \geq 0  \ \mbox{ and } 0\leq i\leq 2t\}.
$$ 
%, where $\zeta_i$ maps basis elements $f^{2t}_j$ to $1$ in $A$ if $i=j$ and $0$ otherwise, 
Then $R$ is a subalgebra of $\HH^{2*}(A)$. It is $\mathbb{Z}_2$-graded with
$R_0= \langle \zeta_i^{2t}: i \mbox{ even}\rangle $ and
$R_1= \langle \zeta_i^{2t}: i \mbox{ odd} \rangle$. Moreover
\begin{align}
\zeta_{l}^{2m}\cdot \zeta_{r}^{2t} = \left\{\begin{array}{ll} 0 & l, r \mbox{ odd} 
\cr \zeta_{l+r}^{2m+2t}& \mbox{otherwise} 
\end{array}
\right.
\end{align}
\end{thm}

As for the case $a=2$ we can see:

\begin{cor} The even part $R_0$ 
of $R$  is isomorphic to the polynomial ring in two variables. 
\end{cor}

%\bigskip

\begin{cor} Assume $A$ is as in the Theorem, and let $\cN$ be the
largest homogeneous nilpotent ideal of $\HH^{2*}(A)$. Then $\HH^{2*}(A)/\cN$ is
isomorphic to $R_0$.
\end{cor}

\medskip

\subsection{Lifting} 
We compute the Yoneda product $\chi \bullet \xi$ where $\chi, \xi$ are $k$-linear combinations of maps $\zeta_j$ as in Definition \ref{def:basisgeneral}.  %To do so, we obtain the lifing maps as before. 

For $\xi$ in the span of the $\zeta_j$, the values of $\xi$ are scalars and therefore they commute with elements of $A^e$. 
Luckily, we are only  interested in the even Hochschild cohomology modulo
homogeneous nilpotent elements.
%
%Let $\xi$ be in the span of the $\zeta_j$, and write $\xi(f_i^{2t}) = p_i$ which is then a scalar. 

Similar as for the case where $a=2$ we use liftings along the minimal projective resolution to define the Yoneda products in the cohomology ring. Let $\xi: P_{2t}\to A$ where 
\begin{align}
\xi(f_i^{2t}) := p_i
\end{align}
and we assume $p_i$ is a scalar multiple of $1$, for all $i$. We restricted to the subring of this kind. %
%The objective is to lift $\xi$ along the projective bimodule resolution %
%(so that we can compute some products in the even part of $\HH^*(A)$). 
Consequently the values $p_i$ commute with all elements in $A^e$. As usual we set $p_i=0$ if $i> 2t$ or if $i< 0$. 

The map $h_0: P_{2t} \to P_0$ is defined by
\begin{align}
h_0(f_i^{2t}):= p_i f_0^0 \qquad (0\leq i\leq 2t).  
\end{align}
Moreover, we search explicit formulae for maps 
\begin{align}
h_s: P_{2t+s} \to P_s. 
\end{align}
For $s\geq 1$ we should have
\begin{align}
h_{s-1}\circ d_{2t+s} = d_s\circ h_{s}. 
\end{align}

\subsubsection{Some formulae in $A^e$}
In order to define these lifting maps $h_s$ for $s>0$ we establish some formulae in $A^e$. Let 
\begin{align}
c_i = 1 + q + \ldots + q^i \qquad \text{ for $0\leq i\leq a-2$}. 
\end{align}

\begin{defn} For an integer $s$ we define
\begin{align}
\beta_x(s)  &=  \sum_{i=0}^{a-2} c_iq^{si}(x^{a-2-i}\otimes x^i) \\
\beta_y(s)  &=  \sum_{i=0}^{a-2} c_iq^{si} (y^i\otimes y^{a-2-i})
\end{align}
%(We may only need $s=1$ and $s=-1$.) 
\end{defn}

Recall now the elements in $A^e$ which occur in the definition of the differentials: 
\begin{align}
\gamma_y(s) = \sum_{j=0}^{a-1}  q^{js}(y^j\otimes y^{a-1-j})\\
\gamma_x(s) = \sum_{j=0}^{a-1} q^{js}(x^{a-1-j}\otimes x^j)
\end{align}
At the end we only need $s=0$ and $s=1$
\begin{align}
&\tau_y(1) = (1\otimes y) - q(y\otimes 1)\\ 
&\tau_x(1) = q(1\otimes x) - (x\otimes 1)\\
&\tau_y(0) = (1\otimes y) - (y\otimes 1) \\ 
&\tau_x(0) = (1\otimes x) - (x\otimes 1). 
\end{align}

%Doing some computations, here are 
%some factorisations: Note that elements with label $y$ commute, and elements with
%label $x$ also commute (ie part of the equations come free). 
%$$\gamma_y(1) \ = \ \beta_y(1)\tau_y(0) \ =  \ \tau_y(0)\beta_y(1)
%\leqno{(1)}$$
%$$-\gamma_x(1) \ = \ \beta_x(1)\tau_x(0)  \ = \ \tau_x(0)\beta_x(1)
%\leqno{(2)}$$
%$$\gamma_y(0) \ = \ \beta_y(1)\tau_y(0) = \tau_y(1)\beta_y(1)
%\leqno{(3)}
%$$
%$$ -\gamma_x(0) \ = \ \beta_x(1)\tau_x(1) \ = \ \tau_x(1)\beta_x(1) \leqno{(4)}
%$$
%$$\gamma_y(0)\beta_x(1) \ = \ \beta_x(-1)\gamma_y(2) \ = \ \beta_x(-1)\beta_y(1)\tau_y(1)
%\leqno{(5)}
%$$
%$$\gamma_x(0)\beta_y(1) = \beta_y(-1)\gamma_x(2) \ = \ -\beta_y(-1)\beta_x(1)\tau_x(1)
%\leqno{(6)}
%$$
%Actually in (5) and (6)  we have  $\gamma_y(2)= \beta_y(1)\tau_y(1)$ and 
%$-\gamma_x(2) = \beta_x(1)\tau_x(1)$.
%$$\beta_x(-1)\beta_y(1) = \beta_y(-1)\beta_x(1)
%\leqno{(7)}
%$$
%To find these is at the end not difficult. One looks at a general term and moves the x's to the same side, and
%then one can see directly what happens. 
%
%
%
%\newpage
%
%\section{Formula for $h_s$}
%
%
%\subsubsection{Formula for the lifting $h_s$}
With this notation, we will define maps $h_s: P_{2t+s}\to P_s$, defined on the generators $f_i^{2t+s}$ of the free $A^e$ module $P_{2t+s}$, and we will show
below that they lift $\xi$: 
\begin{defn}\label{a>2:hs}
\mbox{}\newline
\noindent Assume $s$ is even. For an integer $i$ we define the following elements in the algebra,
\begin{align}
&\omega_+(j) = 
\begin{cases}
\beta_x(-1)\beta_y(1) & \text {$j$ odd} \\
1 & \text {$j$ even}
\end{cases}
\qquad
&\omega_-(j) = 1 
%\begin{cases}
%1 & \text {$j$ odd} \\
%1 & \text {$j$ even}
%\end{cases}
%\left\{\begin{array}{ll} \beta_x(-1)\beta_y(1) & i \mbox{ \ odd \ }\cr
%1 & i \mbox{ \ even } 
%\end{array}\right.
\end{align}
We will show that  a lifting formula is given by
\begin{align}
h_s(f_i^{2t+s}):= 
\begin{cases}
\sum_{j=0}^s p_{i-j}\omega_+(j) f_j^s  & \text{$i$ even} \\
\sum_{j=0}^s p_{i-j}\omega_-(j) f_j^s  & \text{$i$ odd. }  
\end{cases}
%\left\{ \begin{array}{ll} \sum_{j=0}^s p_{i-j}\varepsilon(j) f_j^s  & i  \ \mbox{even}\cr
%& \cr
%\sum_{j=0}^s p_{i-j}f_j^s & i \ \mbox{  odd } 
%\end{array}
%\right.
\end{align}

\noindent
Now assume $s$ is odd. Here we need two parameters in $A^e$,  one for $x$ and one for $y$. We set 
\begin{align}
\varepsilon_x(j) = 
\begin{cases}
-\beta_x(1) & \text{$j$ odd} \\
1           & \text{$j$ even}
\end{cases}
%\left\{\begin{array}{ll} \beta_x(1) & i \ \mbox{ odd}\cr
%1 & i \ \mbox{ even }
%\end{array}
%\right. 
\qquad
\varepsilon_y(j) = 
\begin{cases}
1           & \text{$j$ odd} \\
-\beta_y(0) & \text{$j$ even.}
\end{cases}
%\left\{\begin{array}{ll} 1 & i \ \mbox{ odd }\cr
%\beta_y(1) & i \ \mbox{ even }
%\end{array}
%\right.
\end{align}
We will show that a lifting  formula is given by
\begin{align}
h_s(f_i^{2t+s}):= 
\begin{cases}
\sum_{j=0}^s p_{i-j}\varepsilon_x(j) f_j^s & \text{$i$ even}\\
\sum_{j=0}^s p_{i-j}\varepsilon_y(j) f_j^s & \text{$i$ odd. } 
\end{cases}
%
%\left\{\begin{array}{ll}
%\sum_{j=0}^s p_{i-j}(-1)^j\varepsilon_x(j) f_j^s & i \ \mbox{ even}\cr
%& \cr
%\sum_{j=0}^s p_{i-j}(-1)^{j+1}\varepsilon_y(j) f_j^s & i \  \mbox{ odd}
%\end{array}
%\right.
\end{align}
\end{defn}
\begin{lem}\label{lem:rel}
We have that the following relations hold: 
%page 49-50
\begin{align}
\beta_y(1)\tau_y(1)   &= \gamma_y(2)           \tag{a} \label{rel:a}\\
\beta_x(-1)\gamma_y(2)&= \gamma_y(0)\beta_x(0) \tag{b} \label{rel:b}\\
\beta_x(-1)\beta_y(1) &= \beta_y(-1)\beta_x(1) \tag{c} \label{rel:c}\\
\beta_x(1)\tau_x(1)   &=-\gamma_x(2)           \tag{d} \label{rel:d}\\
\beta_y(-1)\gamma_x(2)&= \gamma_x(0)\beta_y(0) \tag{e} \label{rel:e}\\
\beta_y(0)\tau_y(0)   &= \gamma_y(1)           \tag{f} \label{rel:f}\\
\beta_x(0)\tau_x(0)   &=-\gamma_x(1)           \tag{g} \label{rel:g}\\
\tau_y(1)\beta_y(0)   &= \gamma_y(0)           \tag{h} \label{rel:h}\\
\tau_x(0)\beta_x(-1)  &=-\gamma_x(-1)          \tag{i} \label{rel:i}\\
\gamma_x(-1)\beta_y(1)&= \beta_y(0)\gamma_x(1) \tag{j} \label{rel:j}\\
\tau_y(0)\beta_y(-1)  &= \gamma_y(-1)          \tag{k} \label{rel:k}\\
\tau_x(1)\beta_x(0)   &= -\gamma_x(0)          \tag{l} \label{rel:l}\\
\beta_x(0)\gamma_y(1) &= \gamma_y(-1)\beta_x(1)\tag{m} \label{rel:m}
\end{align}
\end{lem}
\begin{proof}
We prove (\ref{rel:a}) and (\ref{rel:b}), and the other relations follows from the same kind of reasoning. Start with (\ref{rel:a}), we have 
\begin{align}
\beta_y(1)\tau_y(1)&=\left(\sum_{i=0}^{a-2}c_iq^i(y^i \ot y^{a-2-i})\right)\left((1 \ot y) - q(y \ot 1)\right)\\
&=\sum_{i=0}^{a-2}\left(c_iq^i(y^i \ot y^{a-1-i})-c_iq^{i+1}(y^{i+1} \ot y^{a-2-i})\right) \\
&= c_0(1 \ot y^{a-1}) + c_1q(y \ot y^{a-2}) + \cdots + c_{a-2}q^{a-2}(y^{a-2} \ot y) \\ 
&\hphantom{=c_0(1 \ot y^{a-1}),} - c_0q(y \ot y^{a-2})-\cdots- c_{a-3}q^{a-2}(y^{a-2}\ot y)-c_{a-2}q^{a-1}(y^{a-1}\ot 1)\\
&=c_0(1 \ot y^{a-1})+q(c_1-c_0)(y \ot y^{a-2})+q^2(c_2-c_1)(y^2\ot y^{a-3})+\\ 
&\hspace{3cm} \cdots + q^{a-2}(c_{a-2}-c_{a-3})(y^{a-2}\ot y) - q^{a-1}c_{a-2}(y^{a-1}\ot1)
\end{align}
where we have that $c_0=1$, $c_1-c_0=1+q-1=q$ 
\begin{align}
c_{i+1}-c_i=(1+q+\cdots+q^{i+1})-(1+q+\cdots+q^i)=q^{i+1}
\end{align}
We also observe 
\begin{align}
c_{a-1}=1+q+\cdots+q^{a-2}=-q^{a-1}
\end{align}
since $a$ is a root of unity and hence $1+q+\cdots+q^{a-2}+q^{a-1}=0$. Then we have, 
\begin{align}
\beta_y(1)\tau_y(1)&=(1 \ot y^{a-1})+q^2(y \ot y^{a-2})+\cdots+q^{2(a-1)}(y^{a-1}\ot1)=\gamma_y(2)
\end{align}
For the relation (\ref{rel:b}) we inspect a typical element in this sum: 
\begin{align}
c_iq^{-i}(x^{a-2-i}\ot x^{i})q^{2j}(y^{j} \ot y^{a-1-j})=c_iq^{-i}q^{2j}(x^{a-2-i}y^{j} \ot x^{i}*y^{a-1-j})
\end{align}
(where $*$ denotes the multiplication in $A^{\op}$). Now we recall that $xy=qyx$ (and $x*y=q^{-1}y*x$) hence $x^{a-2-i}y^{j}=q^{j(a-2-i)}y^{j}x^{a-2-i}$ and $x^{i}*y^{a-1-j}=q^{-i(a-1-j)}y^{a-1-j}*x^{i}$. We get 
\begin{align}
c_iq^{-i}q^{2j}q^{j(a-2-i)}q^{-i(a-1-j)}(y^{j}x^{a-2-i} \ot y^{a-1-j}*x^{i})=(y^j \ot y^{a-2-j})c_j(x^{a-2-i} \ot x^i)
\end{align}
which is the most typical element in the sum $\gamma_y(0)\beta_x(0)$. 
%Proof of some cases one with both $x$ and $y$ and one with $y$ only, more to come.... 
\end{proof}

The relations (\ref{rel:a}) to (\ref{rel:m}) in Lemma \ref{lem:rel} can be used to prove that the maps $h_s$ are liftings for the given map $\xi$:  

\begin{prop}
The lifting formulas make the suggested squares commutative, that is $h_{s-1} \circ d_{2t+s} = d_s \circ h_{s}$ when $s \geq 1$ and $\xi = \mu \circ h_0$.  
\end{prop}
\begin{proof}
%page 45 
We give details when $s$ and $i$ are even, the other cases are similar. The strategy is to apply both sides to $f_i^{2t+s}$ and express the answer in terms of the basis $\{ f_j^{s-1}\}$, with coefficients in $A^e$ and then show that the coefficients of the $f_j^{s-1}$ in the two expressions are equal.

We have
\begin{align}
(d_s\circ h_s)(f_i^{2t+s}) = &d_s \circ (\sum_{j=0}^s p_{i-j}\omega_+(j)f_j^s)\\
=& \sum_{j\text{ even, } 0\leq j\leq s} p_{i-j}\omega_+(j)[\gamma_y(0)f_j^{s-1} + \gamma_x(0)f_{j-1}^{s-1}]\\
&\hspace{.1cm} +\sum_{j\text{ odd, }0\leq j\leq s} p_{i-j}\omega_+(j)[-\tau_y(1)f_j^{s-1} + \tau_x(1)f_{j-1}^{s-1}]
\end{align}

We split each of the two sums, and when the index is $j-1$ we change variables, setting $l=j-1$ so that $j=l+1$ and noting that $l$ has opposite parity as $j$. As well we set $\omega_+(j)=1$ for $j$ even. Then this becomes
\begin{align}
= & \sum_{j \text{ even, }0\leq j\leq s} p_{i-j}\gamma_y(0)f_j^{s-1}
+   \sum_{l \text{ odd , }-1\leq l\leq s-1}p_{i-l-1}\gamma_x(0)f_l^{s-1}\\
+ & \sum_{j \text{ odd , }0\leq j\leq s} -p_{i-j}\omega_+(j)\tau_y(1)f_j^{s-1}
+   \sum_{l \text{ even, }-1\leq l\leq s-1} p_{i-l-1}\omega_+(l+1)\tau_x(1)f_l^{s-1}
\end{align}
The range of summation can be unified since $f_j^s=0$ for $j=-1$ or $j=s$. We write this now as a combination in the $A^e$-basis $f_j^{s-1}$ for $0\leq j\leq s-1$, (writing $j$ for $l$) and we get 
\begin{align} 
 = & \sum_{j \text{ even, }0\leq j\leq s-1} [p_{i-j}\gamma_y(0) + p_{i-j-1}\omega_+(j+1)\tau_x(1)] f_j^{s-1} \nonumber\\
 +& \sum_{j \text{ odd, }  0\leq j\leq s-1} [p_{i-j-1}\gamma_x(0) - p_{i-j}\omega_+(j)\tau_y(1)]f_j^{s-1}\label{eq:lowerend}
\end{align}

On the other hand
\begin{align}
(h_{s-1}\circ d_{2t+s})(f_i^{2t+s}) = &
h_{s-1} \circ (\gamma_y(0)f_i^{2t+s-1} + \gamma_x(0)f_{i-1}^{2t+s-1})\\
=& \gamma_y(0)[\sum_{j=0}^{s-1}p_{i-j}\varepsilon_x(j)f_j^{s-1}] + \gamma_x(0)[\sum_{j=0}^{s-1} p_{i-1-j}\varepsilon_y(j)f_j^{s-1}]\\\
=& \sum_{j=0}^{s-1} 
[p_{i-j}\gamma_y(0)\varepsilon_x(j) + p_{i-1-j}\gamma_x(0)\varepsilon_y(j)]f_j^{s-1}\label{eq:upperend}
\end{align}

We must show that for each $j$ the coefficients of $f_j^{s-1}$ in (\ref{eq:lowerend}) and in (\ref{eq:upperend}) are equal.

\begin{itemize}
\item[(a)] Assume first $j$ is even. We require
\begin{align}
p_{i-j}\gamma_y(0) + p_{i-j-1}\omega_+(j+1)\tau_x(1) 
= p_{i-1}\gamma_y(0)\varepsilon_x(j) + p_{i-j-1}\gamma_x(0)\varepsilon_y(j)
\end{align}
For $j$ even, $\varepsilon_x(j)=1$ and the first terms agree. The second terms agree provided
\begin{align}
\omega_+(j+1)\tau_x(1) = \gamma_x(0)\varepsilon_y(j) 
\end{align}
Consider the LHS, by identities (\ref{rel:c}), (\ref{rel:d}) and (\ref{rel:e}) it is equal to
\begin{align}
\beta_y(-1)\beta_x(1)\tau_x(1) = -\beta_y(-1)\gamma_x(2) = -\gamma_x(0)\beta_y(0) = \gamma_x(0)\varepsilon_y(j)
\end{align}
from the definition of $\varepsilon_y(j)$ in this case. Hence the second terms agree as well.

\item[(b)] Now assume $j$ is odd. We require 
\begin{align}
p_{i-j-1}\gamma_x(0) - p_{i-j}\omega_+(j)\tau_y(1) = p_{i-j}\gamma_y(0)\varepsilon_x(j) + p_{i-1-j}\gamma_x(0)\varepsilon_y(j)
\end{align}
For $j$ odd, $\varepsilon_y(j)=1$ and the terms with $p_{i-j-1}$ agree. For the other two terms to agree we need 
\begin{align}
\gamma_y(0)\varepsilon_x(j) = -\omega_+(j)\tau_y(1)
\end{align}
We have using the definition and identities (\ref{rel:a}) and (\ref{rel:b}) that 
\begin{align}
-\omega_+(j)\tau_y(1)  
= -\beta_x(-1)\beta_y(1)\tau_y(1) = -\beta_x(-1)\gamma_y(2) = -\gamma_y(0)\beta_x(0)
= \gamma_y(0)\varepsilon_x(j)
\end{align}
as required.
\end{itemize}
\end{proof}
Similar as for the case $a=2$ we define the Yoneda product of the residue classes represented by $\xi$ and $\chi$ to be the residue class represented by the composition 
\begin{align}
\chi \bullet \xi = \chi \circ h_s. 
\end{align}

\subsection{Description of Yoneda products of basis elements when $a\geq3$}\label{subsec:descoflargeyoneda} 

%In Section \ref{subsec:basis_ageq3} we form a basis of $\HH^{2t+s}(A)$.

In the definition \ref{a>2:hs} of the lifting maps,
we have the term $\omega_+(j)= \beta_x(-1)\beta_y(1) \in A^e$ (for $j$ odd).
When this is evaluated  in $A$, it becomes $\omega_+(j)\cdot 1_A$. We claim that this is always zero, in fact $\beta_y(1)\cdot 1_A=0$. 

Namely, we must view $A$ as an $A^e$ bimodule and then 
\begin{align}
\beta_y(1)\cdot 1_A = \sum_{i=0}^{a-2} c_iq^iy^{a-2} = (\sum_{i=0}^{a-2}c_iq^i)y^{a-2}
\end{align}
The following shows that this is zero: 

\begin{lem} 
Let $q$ be a primitive $a$-th root of unity for $a\geq 3$. Let $c_i=1 + q + \ldots + q^i$ for $i\geq 0$, then 
\begin{align}
\sum_{i=0}^{a-2} c_iq^i = 0
\end{align}
\end{lem}
\begin{proof}
Set also $c_{-1} := 0$. Then we have for $i\geq 0$ that $c_i-c_{i-1} = q^i$.  We get 
\begin{align}
\sum_{i=0}^{a-2} c_iq^i =  \sum_i c_i(c_i- c_{i-1})
\end{align}
Therefore (all summations from $i=0$ to $a-2$)
\begin{align*}
(1+q)(\sum_i c_iq^i) = & \sum_{i} c_iq^i + \sum_i c_iq^{i+1}\cr
=& \sum_i c_i(c_i-c_{i-1}) + \sum_i c_i(c_{i+1}-c_i) \cr
=& \sum_{i} (c_ic_{i+1} - c_ic_{i-1}) \cr
=& c_{a-2}c_{a-1} - c_0c_{-1} \cr
=& 0
\end{align*}
since $c_{a-1} = 1 + q + \ldots + q^{a-1}=0$ and $c_{-1}=0$. But $q\neq -1$, so we can cancel by $(1+q)$ and get the claim.
\end{proof}

We analyse now the products, and this will complete the proof of
Theorem \ref{R:a>2}.  
Define 
\begin{align}
R = {\rm Sp}\{ \zeta_i^{2t}: t\geq 0, 0\leq i\leq 2t\}.
\end{align}
We compute products of elements in $R$. 

Let $\chi$ be of degree $2m$ and $\xi$ of degree $2t$, both in $R$. Let $\xi(f_i^{2t}) = p_i\in K$ for $0\leq i\leq 2t$ and $\chi_(f_j^{2m}) = \bar{p}_j \in K$ for $0\leq j\leq 2m$. As before we set $p_i=0$ for $0<i$ of $i>2t$, and similarly we define $\bar{p}_j$ for any $j\in \mathbb{Z}$. 

Then $\chi\bullet \xi$ is the class of $\chi\circ h_{2m}$ where $(h_s)$ is a lifting of $\xi$, where we use the formula computed above. Note that we only need the case when $s=2m$ is even. We have 
\begin{align}
\chi\circ h_s(f_i^{2t+s}) = \chi(\sum_{j=0}^s p_{i-j}\omega_+(j)f_j^s)
= \left\{\begin{array}{ll}
\sum_{j=0}^s p_{i-j}\omega_+(j)\bar{p}_j & i \mbox{ \ even} \cr
\sum_{j=0}^s p_{i-j}\bar{p}_j & i \ \mbox{ odd}
\end{array}
\right.
\end{align}
where we have already used that $\omega_-(j)=1$. 

Now assume $\chi = \zeta_l$ for some $0\leq l\leq s$, so $\bar{p}_l=1$ and $p_j=0$ otherwise. Then the above simplifies to 
\begin{align}
f_i^{2t+s} \mapsto \left\{\begin{array}{ll} p_{i-l}\omega_+(l)\cdot 1 & i \mbox{ even} \cr
p_{i-l}\cdot 1  & i \mbox{ odd}
\end{array}
\right.
\end{align}

Now take $\xi = \zeta_r$ for some $0 \leq r \leq 2t$. Then $p_{i-l}=1$ if $i-l=r$, and $=0$ otherwise. 

Note that $\omega_+(l) \cdot 1_A=0$ for $l$ odd and $=1$ otherwise. The zero occurs precisely when $l$ is odd and $i=l+r$ is even, i.e.\ if both $l, r$ are odd. So we get 
\begin{align}
\zeta_l^{2m}\cdot \zeta_{r}^{2t} = \left\{ \begin{array}{ll} \zeta_{l+r}^{2m+2t} &  l, r \  \mbox{  not both odd }\cr
0 &  l, r \ \mbox{ odd}.
\end{array}
\right.
\end{align}

As for the case $a=2$ we see that $R_0$ is isomorphic to the
polynomial ring in two variables. 

Furthermore, we see that
elements in $R_1$ are also nilpotent. The subalgebra
$R_0$ intersects the largest homogeneous nilpotent
ideal $\cN$ trivially, and the span of the $\eta^{\pm}$ is contained in 
$\cN$.

%%%%%%%%%%%%%%%%%%%%%%%%%%%%%%%%
%%% Acknowledgements 
%%%%%%%%%%%%%%%%%%%%%%%%%%%%%%%%
%\mbox{}
\vspace{\baselineskip}
\noindent\textbf{Acknowledgements.}  
Both authors thank Petter A.\ Bergh for his joint notes with Karin Erdmann. The second author (Magnus Hellstrøm-Finnsen) would like to thank the first author (Karin Erdmann) a lot for the invitation to Oxford and is very grateful for spending these months on this wonderful place. Thanks for valuable discussions on this project and on various topics in general, and for including me into the mathematical community in Oxford. Thanks also to Petter for arranging and help with applications. The Research Council of Norway supports the PhD work of the second author through the research project \emph{Triangulated categories in algebra} (NFR 221893). The Research Council of Norway has also supported the second author with oversea grant in the occasion of the stay in Oxford. 

%%%%%%%%%%%%%%%%%%%%%%%%%%%%%%%%
%%% Bibliography
%%%%%%%%%%%%%%%%%%%%%%%%%%%%%%%%

\printbibliography
%\bibliographystyle{alpha}
%\bibliography{\string~/Documents/Latex/bibFile.bib}
\mbox{}\vfill\mbox{}
\end{document}

%% file: operators.tex
\DeclareMathOperator{\Hom}{Hom}

\DeclareMathOperator{\op}{op}

\DeclareMathOperator{\Ext}{Ext}
\DeclareMathOperator{\Tor}{Tor}

\DeclareMathOperator{\HH}{HH}

\DeclareMathOperator{\im}{im}

\DeclareMathOperator{\ot}{\otimes}